\documentclass[a4paper,12pt,fleqn]{article}

\usepackage{amssymb}
\usepackage{amsmath}
\usepackage{graphicx}
\usepackage{color}
\usepackage{eurosym}
\usepackage{txfonts}
\usepackage{courier}
\usepackage{pifont}

\usepackage{etoolbox}
\let\bbordermatrix\bordermatrix
\patchcmd{\bbordermatrix}{8.75}{4.75}{}{}
\patchcmd{\bbordermatrix}{\left(}{\left[}{}{}
\patchcmd{\bbordermatrix}{\right)}{\right]}{}{}

\patchcmd{\bbordermatrix}{\begingroup}{\begingroup\openup1\jot}{}{}
\makeatletter
\patchcmd{\bbordermatrix}
  {\vcenter{\kern-\ht\@ne\unvbox\z@\kern-\baselineskip}}
  {\vcenter{\kern-\ht\@ne\unvbox\z@\kern-\baselineskip\kern2pt}}
  {}{}
\makeatother

\usepackage[colorlinks=false,urlcolor=blue,ps2pdf=true,
pdfpagemode=UseNone, pdfauthor={Robin Whitty},
pdfsubject={Triangle bisection}, pdfcreator={Robin Whitty},
pdfproducer={Robin Whitty}]{hyperref}

\newtheorem{Thm}{Theorem}
\newtheorem{Def}[Thm]{Definition}
\newtheorem{Prop}[Thm]{Proposition}

\newtheorem{Lm}[Thm]{Lemma}

\newcommand{\vc}[1]{{ \normalfont {\bf #1}}}

\newcommand{\vo}[1]{{ \normalfont {\bf #1}}^{\bot}}
\newcommand {\rl}{\mathbb{R}}

 \newcommand{\ang}[1]{\langle #1 \rangle}
\newcommand{\angg}[2]{\langle #1,#2 \rangle}

\newcommand{\QED}{\unskip\nobreak\hfill\rlap{\llap{\large $\Box$}}} 

\author{ Robin
Whitty\thanks{128 avenue Saint Germier, 31600, France, robinwhitty@www.theoremoftheday.org}}
\date{\today}
\title{All-angles bisection of  polygon area}

\begin{document}

\maketitle

\hspace{.25in}
\parbox[c]{6in}{\small {\bf Abstract} Given a simple polygon we aim to find the equation of the straight line which bisects the area of the polygon in a given direction. Additionally, we would like to vary this direction with  minimal additional calculation.
 We provide a solution in the case where the polygon is `bisection-convex', meaning that any straight line bisecting the area of the polygon contains exactly two points on the boundary of the  polygon.
}


\section{Introduction}
A simple polygon and a direction vector are given  and we want to write down the equation of the line in the given direction which bisects the area of the polygon. This problem was solved very elegantly and in linear time by Shermer \cite{Shermer}. However his solution is given in terms of a vertical bisection, with the given polygon rotated to align this bisection  with the required direction vector. If the vector is changed a new rotation is required and the bisection must be recalculated from scratch.  For some classes of polygons  the totality of bisecting lines may be specified in terms of an envelope to which all bisections are tangent \cite{Berele, DaSilva, Fechtor}. For triangles, this envelope has been known since the late nineteenth century (see, e.g., \cite{Beyer}) and is easy to specify, so that our bisection problem may be solved explicitly as a piece-wise function \cite{Whitty}.  In the current paper we tackle simple polygons in general and  the function becomes a rather elaborate algorithm, which nevertheless avoids explicitly calculating the bisection envelope. Our approach is to make an initial tabulation of data for our polygon which is sufficient to then specify, in sublinear time, a bisection in any direction. This is less general than Shermer: it works for polygons which are `not too' non-convex (`bisection convex' as defined in  \cite{Fechtor}). Nevertheless, in cases where real-time change of bisection angle is what is wanted, our approach, used in tandem with Shermer, seems to us worthwhile. Throughout this paper bisection will always be with respect to area. There are other possibilities, for instance perimeter bisection, which are covered systematically in \cite{Berele}.

\begin{figure}[h]
\hspace{-.3in}\includegraphics[scale=0.30]{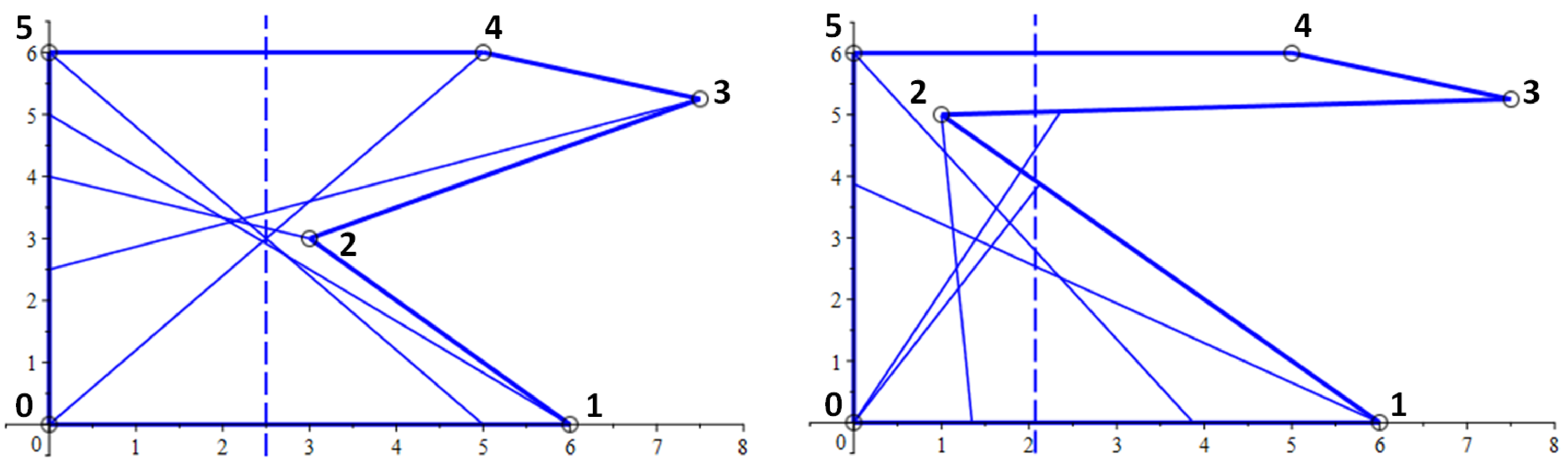}

\hspace{1.1in}(a)\hspace{3.25in}(b)

\caption{(a) A six-vertex bisection-convex polygon. All vertices have integer coordinates except vertex~3 at $(15/2,21/4)$. All chords shown, as well as the vertical dashed line, bisect its area.  (b) A non-bisection-convex polygon. The vertical line and the chords from vertex~1 and vertex~5 are  bisecting. In a specific sense studied in this paper, the other three chords shown also bisect.}
\label{fig:bisectionconvex}
\end{figure}

 \begin{Def}[\cite{Fechtor}]\label{def:bc} A polygon is called {\bf bisection-convex} if any straight line which bisects its area contains exactly two points on the boundary of the polygon.
\end{Def}

Strictly speaking, definition~\ref{def:bc} is referred to by Fechtor-Pradines \cite{Fechtor}, as {\em strictly} bisection convex. He allows also the case where a bisecting line may be tangent to the polygon, that is, may coincide with one or more polygon edges. If you sketch the (non-strictly-bisection-convex) polygon passing anticlockwise through five points, $(0,0),\,(1,0),\,(3,0),\,(1,1),\,(1,2)$, you will find a bisecting  line passing vertically through $(1,0)$ and coinciding with the edge from $(1,1),\,(1,2)$. Such polygons are not beyond the algorithm we present here but their treatment may be safely classed as an implementation issue.

 Figure~\ref{fig:bisectionconvex} shows two non-convex polygons on $n=6$ vertices (labelled  modulo~6). They differ only in the position of vertex~$2$.
  The polygon in figure~\ref{fig:bisectionconvex}(a) is bisection convex. It has area $A_P=30$ so the vertical dashed line, which cuts off a rectangle of half this area, is  bisecting. The six chords shown are all also bisecting. They form a collection in which each vertex is joined by a bisecting line to an opposite edge (in two cases, to the end vertex of an edge). Intuitively, any bisecting line, including the vertical dashed line, must be a pivoting  of one of these chords about a suitable point. This is how our algorithm will work.

 In figure~\ref{fig:bisectionconvex}(b), the vertical dashed line is again bisecting (and is an example of Shermer's algorithm in action): the left and right halfplanes that it creates each contain fifty percent of the polygon area. A selection is shown of chords which join vertices to opposite edges. The chords from  vertex~$1$, at $(6,0)$, and from vertex~5, at $(0,6)$, are bisecting in the halfplane, `global', sense. The chord which descends from the vertex~$2$ at $(1,5)$ is bisecting in a `local' sense: if we descend the chord and then follow the polygon edges anticlockwise back to its start vertex we are describing a region whose area is  half the polygon area. But this is not a bisecting line in the global sense of specifying two halfplanes containing equal shares of the polygon area. The same is true of the chord joining vertex~$0$ at the origin to the edge from vertex~$1$ to vertex~$2$. The other chord from the origin bisects locally in a  more subtle way: descend this chord to the origin and return to its start anticlockwise. Two triangles, one anticlockwise the other clockwise, are traced. Now clockwise areas  are counted as negative. The clockwise triangle has area identical to the thin triangle wedge between the two chords at the origin, so the longer chord again gives a region having  a total of half the polygon area.

 The idea of this paper is that, in bisection convex polygons, the local and global senses of `bisecting' coincide. We will be able to exploit this to collect sufficient local data about the polygon to specify bisections on demand. The word `local' has been used so far in a rather vague way! A formal definition will arrive at the same time as the precise specification of this local bisection data.

 We will be exclusively interested in  bisecting chords as found in figure~\ref{fig:bisectionconvex} so it is worth reserving this term and some notation for them:
 \begin{Def}\label{def:chords}  Let $P$ be a polygon on $n$ vertices, numbered from~0 to~$n-1$ in anticlockwise order, and let $i,j$ be distinct vertices of $P$.
\begin{enumerate}
\item A chord from  vertex $i$ to some point on the edge from $j$ to $j+1$ (taken modulo $n$) will be called a {\bf bisecting chord} if its extension to an infinite line (a) bisects the area of the polygon, and (b) meets the polygon boundary in exactly two points.
 \item The collection of all chords of $P$ joining $i$ to some point on the edge from $j$ to $j+1$  form a triangle which will be denoted $\Delta_{i,j}$. For convenience, we will sometimes refer to $\Delta_{i,j}$ as {\bf the} chord from vertex~$i$ to edge from $j$ to $j+1$. We further abuse notation later by taking $\Delta_{i,j}$ to also denote the area of this triangle.
 \item The triangle $\Delta_{i,j}$ will be referred to as a {\bf bisecting chord} if some individual chord in its collection is a bisecting chord.
 \end{enumerate}
\end{Def}
In figure~\ref{fig:bisectionconvex}(a), we may say that chord $\Delta_{3,5}$ is  bisecting   since the line segment (actual chord) from vertex~3 to point $(0,5/2)$ bisects the polygon and its extension meets the polygon in just these two points. We do not mind that the triangle $\Delta_{3,5}$ also contains line segments joining vertex~3 to the edge between vertices~5 and~0 which fail this test: the line segment to vertex~0 meets the polygon in three points, for example. In figure~\ref{fig:bisectionconvex}(b), chords $\Delta_{1,5}$ and $\Delta_{5,0}$ are  bisecting; in fact all actual chords in these triangles meet the polygon in just two points or, in the case of the line from vertex~1 to vertex~5 and its reverse, a point and an edge (the edge joining vertex~1 and vertex~2). These are, in fact, the only bisecting chords; the other three vertex-edge line segments shown, and two that were omitted for the sake of clarity, all bisect only in the local sense. Such local bisections are not in themselves of interest to us, but their action of tracing positive and negative areas will be essential to the tabulation of polygon data which generates our all-angles bisections.

Our bisecting algorithm  will start by locating the bisecting chord at each vertex of a bisection-convex polygon. Translated to the origin as direction vectors, these chords represent a partition of the half-circle into sectors, see figure~\ref{fig:sectorvectors}. Any desired  slope for bisection will necessarily fall into a unique sector, or will coincide with one of the direction vectors. This identifies a bisecting chord which we may pivot to give a bisection of the desired slope.
\begin{figure}[h]
\centerline{\includegraphics[scale=0.5]{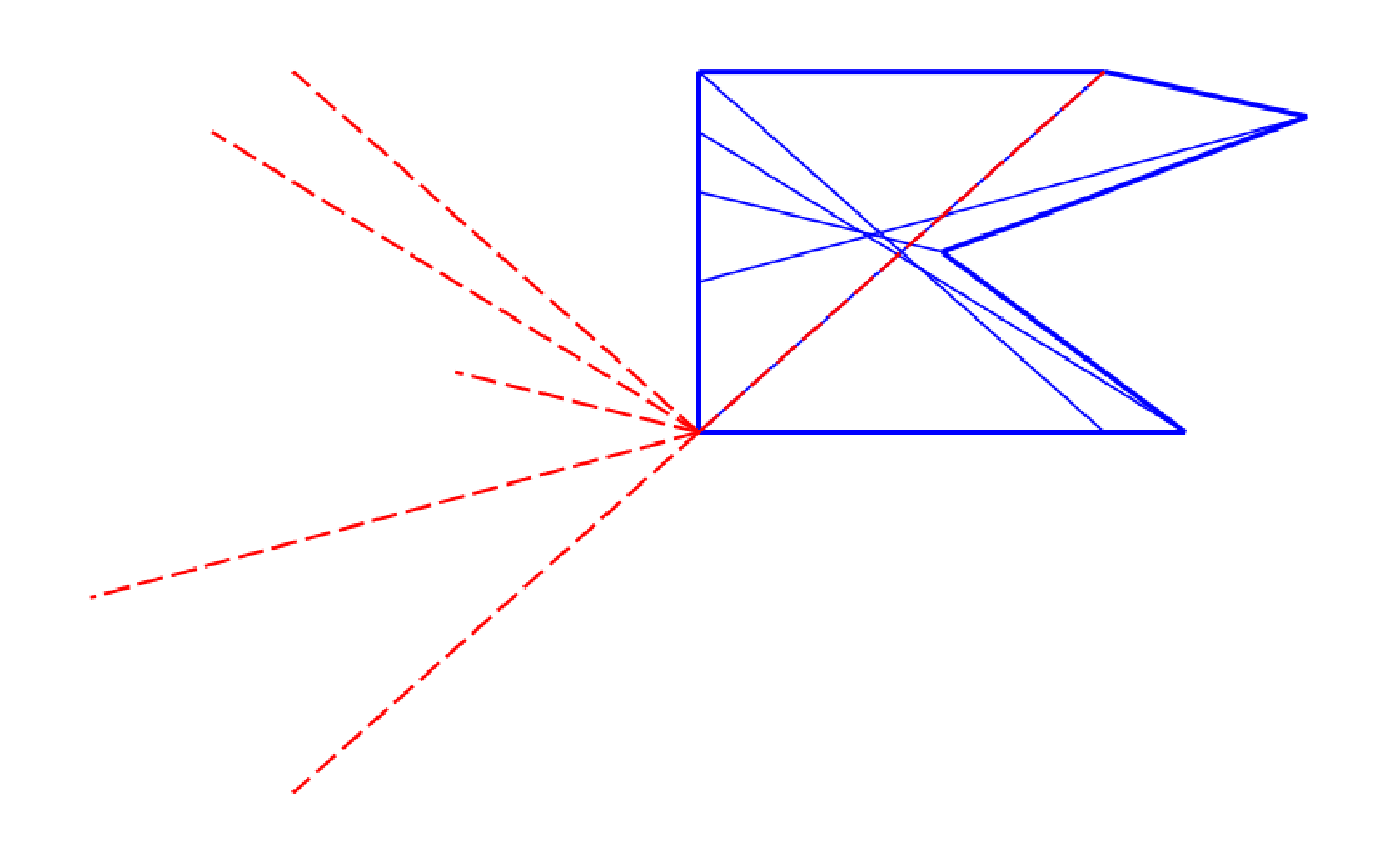}}

\caption{Bisecting chords and (the dashed lines) associated direction vectors}
\label{fig:sectorvectors}
\end{figure}

Finding the bisecting chords and arranging them into sectors of a half-circle is described in the next section. Thereafter, a suitable bisecting chord is `pivoted' to the required slope about a point chosen to guarantee that area bisection is preserved. This is described in section~3. Section~4 addresses the pre-requisite question of how to determine if a given polygon is bisection convex. Finally section~5 summarises the algorithm we have specified and addresses some unresolved issues: can our approach work for non-bisection-convex polygons; and can our local approach to polygon bisection be of value in its own right.

\section{Outline of our method}
\subsection{Our goal: pivoting a bisecting line to a desired angle}
 Let $P$ be an $n$-vertex, simple, bisection-convex polygon. The task in this section is to specify the~$n$ bisecting chords of~$P$ which, taken as direction vectors, will span a total angle of $\tau/2$ radians ($\tau$ being the complete circle). In this sense we may say that each consecutive pair of bisecting chords forms a sector of the half circle. Additionally, a chord must be able to be pivoted about some point, that is turned through an angle less than or equal to the angle formed with its consecutive chord (as a direction vector),  so that its property of bisecting the area of $P$ is conserved.

In figure~\ref{fig:sectors}(a) the bisecting chords $\vc{x}$ and $\vc{y}$ form a sector containing the slope vector $\vc{u}$. The heads and tails of these chords subtend the same two edges, $\vc{a}$ and $\vc{b}$, of $P$ (possibly being end-vertices of these edges). If chord $\vc{x}$ is pivoted (figure~\ref{fig:sectors}(b)) it will form two triangles on edges $\vc{a}$ and $\vc{b}$. By pivoting about the correct point on $\vc{x}$   the areas of these triangles can be made equal and this ensures that the pivoted chord still bisects $P$. It is straightforward to calculate the necessary pivot point, and this is marked in figure~\ref{fig:sectors}(b) as $t$, being the scalar value such that $t\vc{x}$ cuts $\vc{x}$ at  the appropriate  point. This point may now be given as the output of our  algorithm: together with the slope vector $\vc{u}$ it specifies the required bisecting line.
\begin{figure}[h]
\centerline{\includegraphics[scale=0.4]{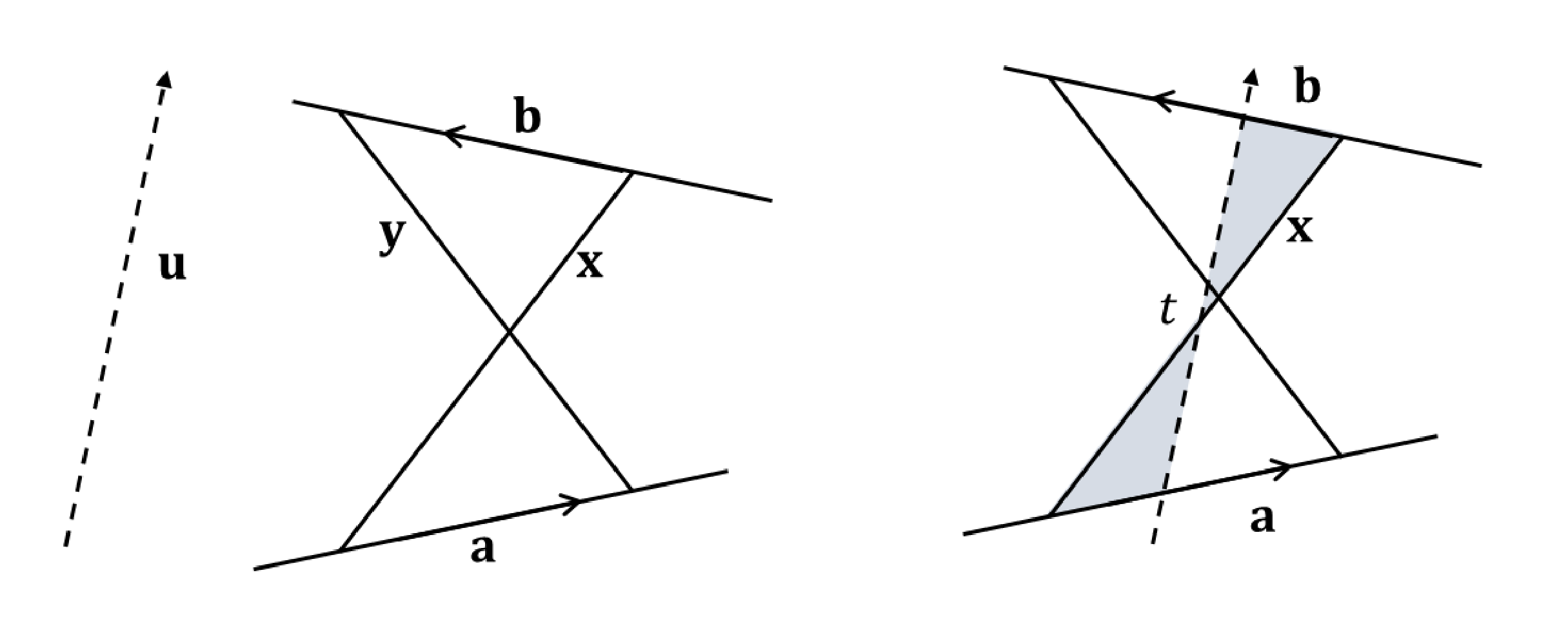}}

\hspace{1.6in}(a)\hspace{3in}(b)

\caption{(a) Choosing two consecutive bisecting chords $\vc{x}$ and $\vc{y}$ for the given slope vector $\vc{u}$; (b) pivoting on $\vc{x}$ to bisect in direction $\vc{u}$}
\label{fig:sectors}
\end{figure}

\subsection{Bisecting chords: identification strategy}
\label{sub:bisecting-chords-id}
We return to the details of the pivoting calculation in section~3. For now, figure~\ref{fig:sectors}(b) serves to suggest how to create our collection of bisecting chords: for the pivoting mechanism to work it will be sufficient to  find a bisecting chord from each vertex of polygon~$P$ to some opposite edge. These chords may then be reordered and reoriented to become consecutive direction vectors, as shown in figure~\ref{fig:sectorvectors}. This strategy may be justified as follows.
\begin{Lm}
\label{lm:intersecting}
For any polygon, two bisecting lines must  intersect in the interior of its convex hull.
\end{Lm}
{\bf Proof.} The intersection of the two lines divides the polygon's area into four regions. Pairs of opposite regions must have equal area since the lines are bisecting. But if the point of intersection is exterior to the convex hull, one of these regions will contain zero polygon area, which is impossible.
\QED

\begin{Lm}
\label{lm:circling}
Given a bisection-convex polygon on $n$ vertices its $n$ bisecting chords may be taken in order so that the angles between consecutive pairs, together with the angle between the first and last, total $\tau/2$ radians.
\end{Lm}
{\bf Proof.} By lemma~\ref{lm:intersecting}, the collection of all points of intersection of the $n$ bisecting chords are contained in some smallest circle such that the bisecting chord from vertex~$0$ forms a diameter. The remaining  bisecting chords intersect the circle at $n-1$ points lying on the same side of this diameter. Draw lines from the centre of the circle to these points. The collection of angles between lines in the circle total  $\tau/2$ radians.
\QED

\begin{Thm}
\label{Thm:bisection-convex}
A polygon $P$ on $n$ vertices is bisection-convex if and only if it has $n$ bisecting chords.
\end{Thm}
{\bf Proof.} The `only if' part is immediate from the definition of bisection convexity, defintition~\ref{def:bc}.

For the `if' part,  suppose that a bisecting line $L$ is given. We must show that it bisects~$P$ while remaining interior to, or on the boundary of, $P$. We may locate $L$ on the half circle of bisecting chords, as in lemma~\ref{lm:circling}. Suppose the bisecting chord preceding $L$, anticlockwise, is $c$ and the bisecting chord following it is $c'$. By Lemma~\ref{lm:intersecting},~$c$ and~$c'$ intersect in the interior of the convex hull of~$P$, so we may denote the points of intersection of these chords with the boundary of~$P$ by $i,i',j,j'$, in anticlockwise order, where~$c$ joins~$i$ to~$j$ and~$c'$ joins~$i'$ to~$j'$. Suppose there is a vertex~$v$ of~$P$ lying strictly between~$i$ and~$i'$ anticlockwise on~$P$. Then, since there are $n$ bisecting chords, there is a bisecting chord originating at~$v$, contradicting the choice of~$c$ and~$c'$ as consecutive. Similarly there can be no vertex lying anticlockwise between~$j$ and~$j'$. On the other hand, $L$ must intersect~$P$ between~$i$ and~$i'$ and again between~$j$ and~$j'$, by the choice of~$c$ and~$c'$. Say these intersections occur at~$x$ and~$x'$, respectively. Now suppose that $L$ meets the boundary of $P$ in some further point~$y$ which, without loss of generality, we will locate on the continuation of $L$ beyond  $x'$ Then the edges of~$P$ must form a clockwise angle somewhere between~$x'$ and~$y$. This cannot lie anticlockwise beyond~$j'$, otherwise bisecting chord~$c'$ will also intersect the boundary of~$P$ additionally to the intersections at~$i'$ and~$j'$. But if the angle occurs between~$i'$ and~$j'$ then this requires there to be a vertex of~$P$ between~$j$ and~$j'$ which we know is impossible. Therefore~$L$ intersects~$P$ only in two points.
\QED

We observe that Theorem~\ref{Thm:bisection-convex} provides an algorithm for deciding whether a given polygon is bisection-convex. It is this algorithm which we investigate in section~4. Happily, it runs off the same  initial polygon data tabulated for our bisection algorithm: a matrix $\Delta_P$ of  areas of triangles joining polygon vertices to opposite edges, and a matrix $R_P$ of chords which are bisecting in the  local sense which we illustrated in figure~\ref{fig:bisectionconvex}(b).

\section{Preprocessing: assembling polygon data}
\subsection{The triangle areas matrix}
\label{subsection:triangleareas}
 The edges of $P$ may be written $[i,i+1]$, counting modulo~$n$ and using square brackets to avoid any confusion with coordinate pairs. Similarly, we may refer to non-polygon edges joining vertices of $P$: $[i,i+2]$,$[i,i-2]$, etc. We may refer to vertex $i$, for short, but in calculations this is represented by a position vector $\vc{v}_i$. The edge $[i,i+1]$ is $-\vc{v}_i+\vc{v}_{i+1}$ as a direction vector. The position vector of a point on edge $[i,i+1]$ is  given by $\vc{v}_i(1-r)+\vc{v}_{i+1}r$ where $r\in [0,1]$.

Our polygon data is based on  triangles formed on $P$ by letting a vertex $i$ subtend a non-incident edge $[j,j+1]$. Following definition~\ref{def:chords}, this is our  triangle $\Delta_{i,j}$ of line segments whose area it is convenient to also denote by $\Delta_{i,j}$. With the convention that $\Delta_{i,i}=0$ (no line seqments) and $\Delta_{i,i-1}=0$ (one line segment, zero area), these $\Delta_{i,j}$ form a matrix, denoted by $\Delta_P$.
\begin{figure}[h]
\hspace{.1in}\parbox[c]{3in}{\includegraphics[scale=0.32]{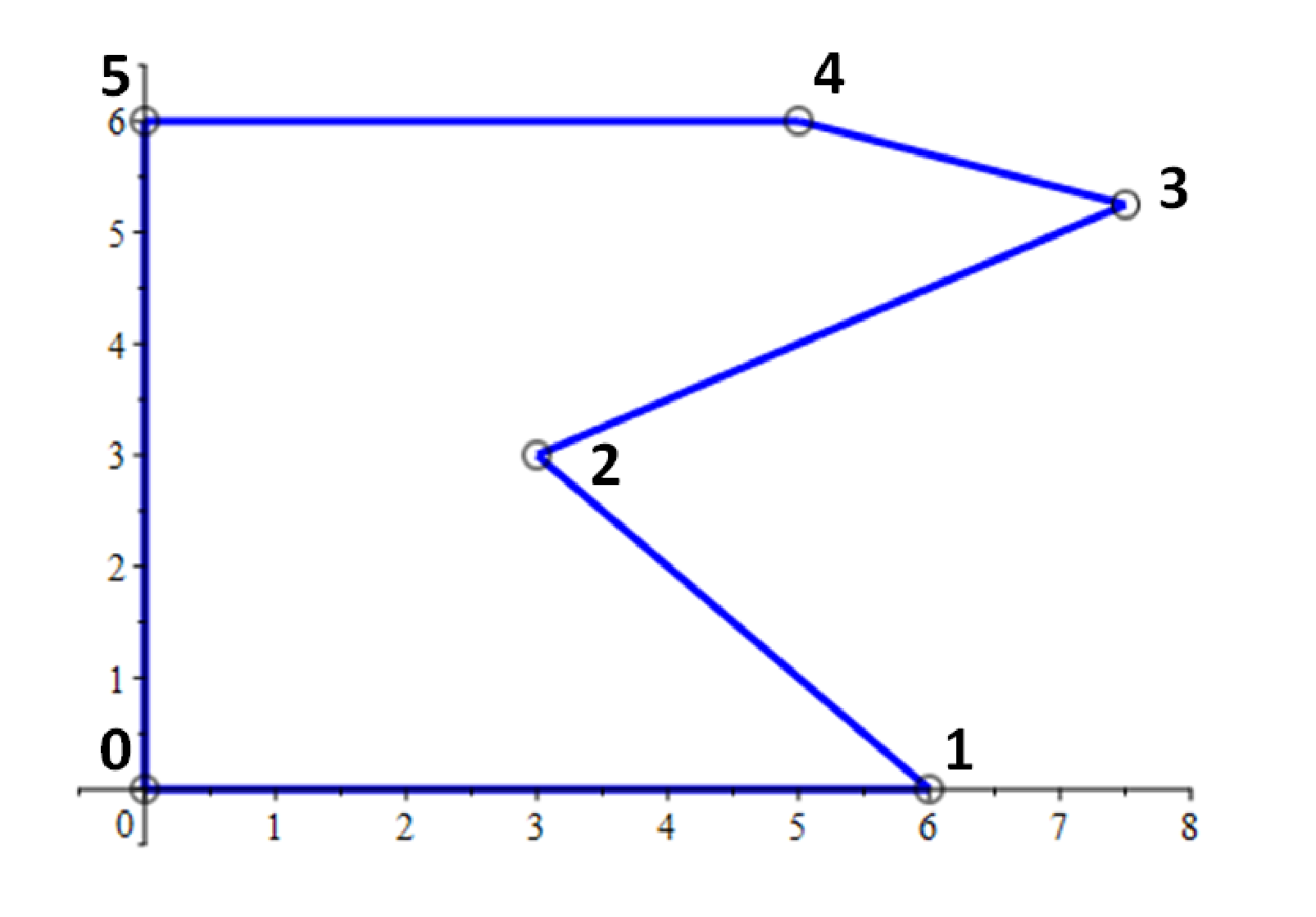}}\hspace{.3in}\parbox[c]{3in}{$\Delta_P=\bbordermatrix{ & 0 & 1 & 2 & 3 & 4 & 5  \cr
0 & 0&9&-{\frac{27}{8}}&{\frac{75}{8}}&15&0 \cr
1 & 0&0&-{\frac{81}{8}}&{\frac{57}{8}}&15&18\cr
2 & 9&0&0&{\frac{9}{2}}&{\frac{15}{2}}&9\cr
3 & {\frac{63}{4}}&-{\frac{81}{8}}&0&0&{\frac{15}{8}}&{\frac{45}{2}}\cr
4 & 18&-{\frac{15}{2}}&{\frac{9}{2}}&0 &0&15\cr
5 & 18&0&{\frac{81}{8}}&{\frac{15}{8}}&0&0
}
$}

\caption{Bisection-convex polygon of figure~\ref{fig:bisectionconvex}(a), and its triangle area matrix $\Delta_P$.}
\label{fig:triangles}
\end{figure}

We make some observations which we will illustrate with reference to figure~\ref{fig:bisectionconvex}(a) as reproduced in figure~\ref{fig:triangles}:
\begin{itemize}
\item Unless $P$ is convex,  triangles formed by vertices subtending edges will not necessarily be entirely contained within $P$. In figure~\ref{fig:triangles} we have
area $\Delta_{1,4}=15$ which comprises a triangle outside of $P$ (apex at vertex~1 and base on edge $[2,3]$) whose area is $135/26$ and a quadrilateral inside $P$ whose area is $255/26$.
\item Three or more vertices collinear will result in zero entries in  $\Delta_P$. Thus we have
$\Delta_{5,1}=0$ because vertices $5,1$ and~$2$ are collinear.
\item Area respects the anticlockwise ordering of vertices, meaning that triangles whose vertices are ordered clockwise are taken to have negative area. Thus $\Delta_{1,2}=-81/8$.
\end{itemize}

The matrix $\Delta_P$ provides the raw data needed for constructing our second tabulation, a matrix $R_P$ of locations on edges which bisect $P$ in the local sense. When $P$ is bisection-convex, entries in $R_P$ which lie in the interval $[0,1]$ then identify bisections which are also global (halfplane) bisections. Our matrices are further used (in section~4) to test if $P$ is bisection-convex, via Theorem~\ref{Thm:bisection-convex}. Given its central role in our algorithm it counts as a surprising piece of good luck that the whole of the triangle areas matrix can be extrapolated from a suitable choice of just three initial rows:
\begin{Thm}
\label{Thm:DeltaRank}{\cite{Forbes}}
For any polygon $P$ the triangle areas matrix $\Delta_P$ has rank~3.
\end{Thm}
\QED

Recall (definition~\ref{def:chords} again) that the `chord' $\Delta_{i,j}$ is bisecting if it includes an actual chord of $P$ which bisects the area of $P$. Although, in general, $\Delta_{i,j}$ will not be bisecting, there will be some point on the extension of edge $[j,j+1]$ which receives a line segment from vertex $i$ which is bisecting in the local sense. We need to identify such line segments in order to check if they fall within $\Delta_{i,j}$. So let $\angg{j}{j+1}$ denote the extension of edge $[j,j+1]$ as an infinite straight line.

We are acquiring rather a lot of  notation! In the polygon in figure~\ref{fig:triangles}, $(5,0)$ gives the coordinates of a point on polygon edge $[0,1]$ joining vertices~$0$ and~$1$; $[5,0]$ is the edge from vertex~$5$ to vertex~$0$; and $\angg{5}{0}$ is the infinite vertical line through this edge. We are allowing ourselves to write $\Delta_{5,0}=18$ because that is the area of triangle joining vertex~$5$ to edge $[0,1]$. But also, $\Delta_{5,0}$ is this triangle as the collection of all chords joining vertex~$5$ to points on edge $[0,1]$. In this sense, $\Delta_{5,0}$ is  a bisecting chord because the chord from vertex $5$ to point $(5,0)$ on edge $[0,1]$ is bisecting.

 The triangle $\Delta_{5,2}$ does not qualify as a bisecting chord because no chord from vertex~$5$ to edge $[2,3]$ can cut off half the area of the polygon. But there will be some point on $\angg{2}{3}$ that is the target of a (locally) bisecting line from vertex~$5$. Indeed $(5/3,7/3)$ is this point: the whole polygon has area~$30$ and the pentagon on point $(5/3,7/3)$ and vertices~$2, 3,4$ and~$5$ has area~$15$. The specification of such points of local bisection will be recorded in our matrix $R_P$.

\subsection{The local bisections matrix}
\label{subsection:local-bisections}
For distinct vertices $i,j$, we examine the infinite line $\angg{j}{j+1}$ to find the point where the line from vertex~$i$ becomes bisecting (in the local sense). Recall that vertex~$k$ is shorthand for  the position vector $\vc{v}_k$. For real parameter $r$ denote by $\vc{v}_j\ang{r}$ the weighted sum $\vc{v}_j(1-r)+\vc{v}_{j+1}r$. If $r\in[0,1]$ then $\vc{v}_j\ang{r}$ will lie on the actual edge $[j,j+1]$. In particular, $\vc{v}_j\ang{0}=\vc{v}_j$ and $\vc{v}_j\ang{1}=\vc{v}_{j+1}$. If $r\in[0,1]$ and line segment $\vc{v}_i$ to $\vc{v}_j\ang{r}$ is bisecting then this is when we say that $\Delta_{i,j}$ is a bisecting chord.

The translation of the data in the triangle areas matrix $\Delta_P$ into a second matrix $R_P$, which we call the local bisections matrix, begins by defining two polygons associated with $\angg{i}{j}$.
\begin{itemize}
\item $P^{+}(i,j,r)$ is formed by traversing straight line segments $v_i$ to $v_j\ang{r}$ and $v_j\ang{r}$ to $v_{j+1}$ and then following the edges of $P$  anticlockwise  from $v_{j+1}$ to $v_i$.
\item $P^{-}(i,j,r)$ is formed by traversing straight line segment $v_j\ang{r}$ to $v_i$ and then following the edges of $P$  anticlockwise  from $v_i$ to $v_j$ and finally following the straight line segment $v_j$ to $v_j\ang{r}$.
\end{itemize}
For a convex polygon  $P$, and for $r\in[0,1]$, the polygons $P^{+}$ and $P^{-}$ will partition the  area of~$P$. In general this will not be the case and the two  polygons need not even be simple since the edge sequences forming them may self-intersect. Nevertheless, we will assign areas to the two polygons in a formal sense as sums of cross products, and these areas will indeed partition the area of $P$, again in a formal sense.

 For position vectors $v_i=(x_i,y_i)$ and $v_j=(x_j,y_j)$, denote by $v_iv_j$ the cross product of $v_i$ and $v_j$ in the plane: $v_iv_j=(x_i,y_i)\times(x_j,y_j)=x_iy_j-x_jy_i$. The following formula (due to Gauss \cite{Whitty2}) is well-known but we will state it for ease of reference:
\begin{Thm}[Shoelace Formula]
\label{thm:shoelace}
Let $P$ be a simple polygon on $n$ vertices, $v_0,v_1, \ldots, v_{n-1}$, given in anticlockwise order and specified as position vectors (i.e, coordinate pairs).  Then the area $A_P$ of $P$ is given by
\begin{equation}
\label{eqn:shoelace}
A_P=\frac12\sum_{i=0}^{n-1}v_iv_{i+1},
\end{equation}
(the indices being taken mod $n$).
\end{Thm}
 \QED

In the case of non-simple polygons the sum of products in equation~\ref{eqn:shoelace} can be interpreted as summing areas over regions weighted by their winding numbers (see, for example, \cite{Seidel}). It will be sufficient for our purposes to assign area values, as formal calculations, to our polygons $P^+(i,j, r)$ and $P^-(i,j, r)$, respectively, as follows:
 \begin{eqnarray}
 \label{eqn:AplusAminus1}
  A^+(i,j, r)&=&\frac12\left(v_iv_j\ang{r}+v_j\ang{r}v_{j+1}+\sum_{k=1}^{i-j-1} v_{j+k}v_{j+k+1}\right),\\
  \label{eqn:AplusAminus2}
  A^-(i,j, r)&=&\frac12\left(v_j\ang{r}v_i+\sum_{k=0}^{j-i-1} v_{i+k}v_{i+k+1}+v_jv_j\ang{r}\right),
\end{eqnarray}
(the indices being taken mod $n$).

Our next result says that these two formal areas, for any value of~$r$, sum to the area of $P$, while their difference, calculated at $r=0$ and suitably scaled, specifies the point on $\angg{j}{j+1}$ at which the two areas become equal. This equality, finally, is what we define to be local bisection of~$P$, and it corresponds to global, half plane, bisection when the point specified lies on edge $[j,j+1]$ and $P$ is bisection-convex.
\begin{Thm}
\label{thm:AplusAminus}
For any polygon $P$, let $A_P$ denote its area. Let $P^{+}(i,j,r)$ and $P^{-}(i,j,r)$ be defined as above, and let $A^{+}(i,j,r)$ and $A^{-}(i,j,r)$ be their respective areas as defined by equations~(\ref{eqn:AplusAminus1}) and~(\ref{eqn:AplusAminus2}). Then
\begin{enumerate}
\item  for any value of $r\in \rl$, $A^{+}(i,j,r)+A^{-}(i,j,r)=A_P$;
\item  for the value
\begin{equation}
\label{eqn:rho}
r_{i,j}= \frac{A^{+}(i,j,0)-A^{-}(i,j,0)}{2\Delta_{i,j}},
\end{equation}
we have $A^{+}(i,j,r_{i,j})=A^{-}(i,j,r_{i,j})=A_P/2.$
\end{enumerate}
\end{Thm}

We will prove this theorem after giving some examples. Refer again to figure~\ref{fig:triangles} in which the polygon has area $30$. Consider the chords joining vertex~$5$ to edge  $[2,3]$, constituting the triangle $\Delta_{5,2}$ with area~$\Delta_{5,2}=81/8$. Now $P^{+}(5,2,0)$ is the quadrilateral on vertices $5,2,3$ and~$4$, whose area is $A^{+}(5,2,0)=12$ (indeed, it is the sum of $\Delta_{5,2}$ and $\Delta_{5,3}$). And $P^{-}(5,2,0)$ is  the quadrilateral on vertices $2,5,0$ and~$1$, whose area is $A^{+}(5,2,0)=18$. By equation~(\ref{eqn:rho}), $r_{5,2}=(12-18))/(2\times 81/8)=-8/27$. Thus, the local bisection point for a line segment from vertex~$5$ to extended edge $\angg{2}{3}$ is $\vc{v}_2\ang{r_{5,2}}=(3,3)\cdot\left(1-r_{5,2}\right)+(15/2,21/4)\cdot r_{5,2}=(5/3,7/3).$ This is the point we identified at the end of section~\ref{subsection:triangleareas} as locally bisecting for vertex~$5$ on edge $[2,3]$.

As an example of a non-simple $P^-$ polygon, consider the local bisection for vertex~$5$ on extended edge $\angg{3}{4}$. We calculate the value $r_{5,3}=-7$, placing the local bisection point at the intersection of $\angg{3}{4}$ with the horizontal axis, at $(25,0)$. Now $P^+(5,3,-7)$ is a triangle  with height~6 and base~5 (apex at $(25,0)$, base  $[4,5]$) so $A^+(5,3,-7)=15$. And $P^-(5,3,-7)$ comprises an anticlockwise non-convex pentagon followed by a clockwise triangle with $A^-(5,3,-7)=1845/74-735/74=15$.

{\bf Proof of theorem~\ref{thm:AplusAminus}}
 Recall that $v_iv_i=v_i\times v_i=0$ and $v_iv_j+v_jv_i=v_i\times v_j+v_j\times v_i=0$; also our notation $\vc{v}_j\ang{r}=\vc{v}_j(1-r)+\vc{v}_{j+1}r$.
\begin{enumerate}\item From equations~(\ref{eqn:AplusAminus1}) and~(\ref{eqn:AplusAminus2}), simplifying notation by multiplying through by 2:
\begin{eqnarray*} 2(A^{+}(i,j,r)+A^{-}(i,j,r))&=&  v_iv_j\ang{r}+v_j\ang{r}v_{j+1}+v_{j+1}v_{j+2}+\ldots +v_{i-1}v_i  \\
                                           & &+ v_j\ang{r}v_i+v_iv_{i+1}+ \ldots +v_{j-1}v_j+v_jv_j\ang{r}\\
                                           &=& v_j\ang{r}v_{j+1}+v_{j+1}v_{j+2}+\ldots +v_{i-1}v_i+v_iv_{i+1}+ \ldots +v_{j-1}v_j +v_jv_j\ang{r} \\
                                           &=&  \left(v_j(1-r)+v_{j+1}r\right)\times v_{j+1}+\left(\sum_{k=j+1}^{j-1}v_kv_{k+1}\right)+v_j\times\left(v_j(1-r)+v_{j+1}r\right) \\
                                           &=& \left(\sum_{k=j+1}^{j-1}v_kv_{k+1}\right)+v_j(1-r)v_{j+1}+v_jv_{j+1}r \\
                                           &=& \sum_{k=j+1}^{j}v_kv_{k+1} =2A_P \mbox{, by theorem~\ref{thm:shoelace}}.
\end{eqnarray*}
\item A similar calculation gives
\begin{eqnarray}\label{eqn:AA1}    2A^{+}(i,j,r) &=&  v_iv_j(1-r)+v_iv_{j+1}r+v_jv_{j+1}(1-r)+ \sum_{k=j+1}^{i-1}v_kv_{k+1} \\
                 \label{eqn:AA2}  -2A^{-}(i,j,r))&=&  -v_jv_i(1-r)-v_{j+1}v_ir+v_jv_{j+1}r- \sum_{k=i}^{j-1}v_kv_{k+1}
\end{eqnarray}
 Add (\ref{eqn:AA1}) and (\ref{eqn:AA2}) and set the result equal to zero:
 \begin{eqnarray*}
                               0   &=& r(-v_iv_j+v_iv_{j+1}-v_jv_{j+1})+v_iv_j+v_jv_{j+1}  + \sum_{k=j+1}^{i-1}v_kv_{k+1} \\
                                          & & +r(v_jv_i-v_{j+1}v_i-v_jv_{j+1})-v_jv_i-  \sum_{k=i}^{j-1}v_kv_{k+1}\\
                     \mbox{Solve for $r$}:&&\\
                                         r  &=& \frac{v_iv_j+v_jv_{j+1}-v_jv_i+  \sum_{k=j+1}^{i-1}v_kv_{k+1}-\sum_{k=i}^{j-1}v_kv_{k+1}}{2v_iv_j+2v_{j+1}v+i+2v_jv_{j+1}}  \\
                                         &=&\frac{v_iv_j+v_jv_{j+1}+\sum_{k=j+1}^{i-1}v_kv_{k+1}-\left(v_jv_i +\sum_{k=i}^{j-1}v_kv_{k+1}\right)}{4\Delta_{i,j}}\\
                                         &=&\frac{A^+(0,i,j)-A^-(0,i,j)}{2\Delta_{i,j}}.
\end{eqnarray*}
\end{enumerate}
\QED

\subsection{Populating the local bisections matrix}
Our strategy for identifying bisecting chords, then, is to take each polygon vertex in turn and test each chord from that vertex in turn to identify which is bisecting. To do this we need to calculate all the corresponding $r_{i,j}$ values to identify which lie in the interval $[0,1]$. This would be very laborious if we had to repeatedly apply equation~(\ref{eqn:rho}), finding all  the necessary auxiliary polynomials $P^{+}$ and $P^{-}$ and their areas!  Luckily, cycling around the polygon, these areas are closely inter-related by  recurrences specified using the triangle areas~$\Delta_{i,j}$:

\begin{Prop}
\label{prop:feedforward}
Let $P$ have area $A_P$. For $i=0,\ldots,n-1$ and $j=i+1,\ldots,i+n-2$,
\begin{equation}
r_{i,j}=\left\{\begin{array}{ccl}
A_P/2\Delta_{i,j}&&\mbox{if }j=i+1 \mbox{ and }  \Delta_{i,j}\neq 0 \\
(r_{i,j-1}-1)\Delta_{i,j-1}/\Delta_{i,j}&&\mbox{if }j\geq i+2\mbox{ and both }\Delta_{i,j} \mbox{ and }\Delta_{i,j-1} \mbox{ are nonzero}.
\end{array}
\right.
\end{equation}
\end{Prop}

{\bf Proof.}
Let $j=i+1$ and suppose that $\Delta_{i,j}\neq 0$. Then $P^{+}(i,j,0)$ is the whole of $P$, while $P^{-}(i,j,0)$ is the empty polygon. So $\displaystyle r_{i,j}=\frac{A^{+}(i,j,0)-A^{-}(i,j,0)}{2\Delta_{i,j}}=(A_P-0)/(2\Delta_{i,j})$.

Suppose that $j\geq i+2$ and $\Delta_{i,j-1},\Delta_{i,j}\neq 0$, so $r_{i,j-1}$ and $r_{i,j}$ are defined. Let $A^{\pm}_{x,y}$ denote $A^{\pm}(x,y,0)$. Despite their formal definitions in equations~(\ref{eqn:AplusAminus1}) and~(\ref{eqn:AplusAminus2}) these areas may be treated as Euclidean areas because they are defined in terms of edge $[j,j+1]$ rather than extended edge $\angg{j}{j+1}$.

There are two configurations, according to the sign of $\Delta_{i,j-1}$:
\begin{description}
\item[\boldmath $\Delta_{i,j-1}>0$:]
 Then $A^{+}_{i,j}=A^{+}_{i,j-1}-\Delta_{i,j-1}$ and $A^{-}_{i,j}=A^{-}_{i,j-1}+\Delta_{i,j-1}$. So
 $$r_{i,j}=\frac{A^{+}_{i,j}-A^{-}_{i,j}}{2\Delta_{i,j}}=\frac{(A^{+}_{i,j-1}-\Delta_{i,j-1})-(A^{-}_{i,j-1}+\Delta_{i,j-1})}{2\Delta_{i,j}}=(r_{i,j-1}-1)\Delta_{i,j-1}/\Delta_{i,j}.$$
 \item[\boldmath$\Delta_{i,j-1}<0$:] Then $A^{+}_{i,j}=A^{+}_{i,j-1}+|\Delta_{i,j-1}|=A^{+}_{i,j-1}-\Delta_{i,j-1}$. And $A^{-}_{i,j}-A^{-}_{i,j-1}=\Delta_{i,j-1}$, and the calculation becomes the same as case~1.
\end{description}
\QED

We require a supplement to Proposition~\ref{prop:feedforward} to take account of zero values among the $\Delta_{i,j}$,  the result of three or more  vertices being collinear. The following says that such zero-area triangles may be `skipped over':
\begin{Prop}
\label{prop:skip}
Suppose that $\Delta_{i,j}=0$ but that $\Delta_{i,j-1}$ and $\Delta_{i,j+1}$ are both nonzero. Then $$r_{i,j+1}=(r_{i,j-1}-1)\Delta_{i,j-1}/\Delta_{i,j+1}.$$
\end{Prop}
{\bf Proof.} We give an informal argument `in the limit', calculating $r_{i,j+1}$ from $r_{i,j-1}$, applying Proposition~\ref{prop:feedforward} to a `perturbation' of our polynomial in which $\Delta_{i,j}$ behaves as a variable.
\begin{eqnarray*}
r_{i,j+1}&=& (r_{i,j}-1)\frac{\Delta_{i,j}}{\Delta_{i,j+1}}\\
&=&\left((r_{i,j-1}-1)\frac{\Delta_{i,j-1}}{\Delta_{i,j}}-1\right)\frac{\Delta_{i,j}}{\Delta_{i,j+1}}\\
&=&r_{i,j-1}\frac{\Delta_{i,j-1}}{\Delta_{i,j+1}}-\frac{\Delta_{i,j-1}}{\Delta_{i,j+1}}-\frac{\Delta_{i,j}}{\Delta_{i,j+1}}\\
&=&(r_{i,j-1}-1)\frac{\Delta_{i,j-1}}{\Delta_{i,j+1}}, \mbox{ as }\Delta_{i,j}\rightarrow 0.
\end{eqnarray*}
\QED

(A more formal proof may be written down by investigating the polygon areas $A^{+}$ and $A^{-}$ involved, as in the proof of proposition~\ref{prop:feedforward}.)

We can illustrate the use of proposition~\ref{prop:feedforward} using the  bisection-convex polygon of figure~\ref{fig:bisectionconvex}(a), reproduced in figure~\ref{fig:rho-vals} together with its matrix $R_P$ of $r_{i,j}$ values.
\begin{figure}[h]
\hspace{.1in}\parbox[c]{3in}{\includegraphics[scale=0.30]{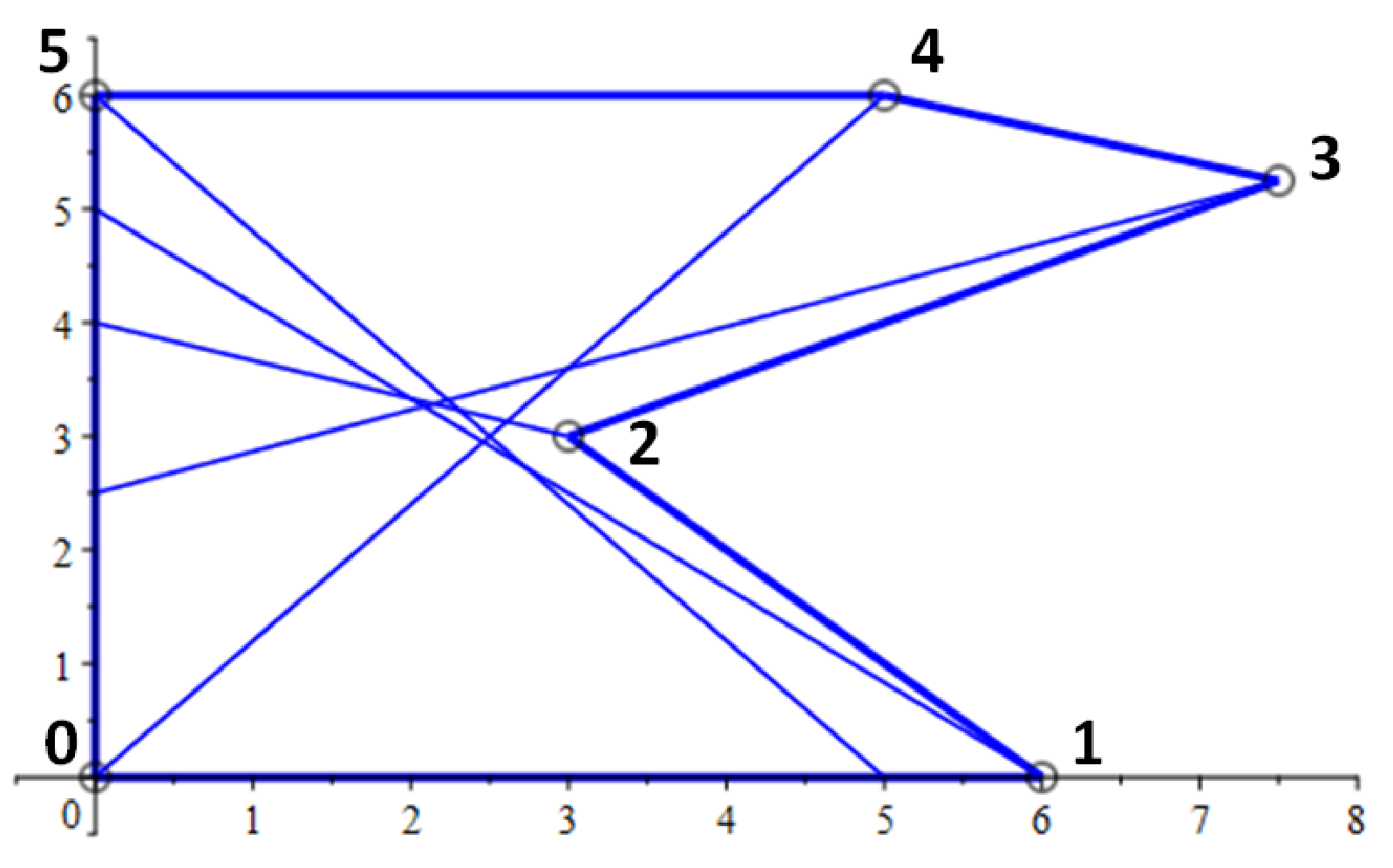}}\hspace{.3in}\parbox[c]{3in}{$R_P=\bbordermatrix{ & 0 & 1 & 2 & 3 & 4 & 5  \cr
0 & 0&{\frac{5}{3}}&-{\frac{16}{9}}&1&0&0
\cr
1& 0&0&-{\frac{40}{27}}&{\frac{67}{19}}&{\frac{6}{5}
}&{\frac{1}{6}}\cr
2 & -{\frac{2}{3}}&0&0&{\frac{10}{3}}&
{\frac{7}{5}}&{\frac{1}{3}}\cr
3 & -{\frac{25}{42}}&{
\frac{67}{27}}&0&0&8&{\frac{7}{12}}\cr
4 & 0&{\frac{12}{5
}}&-{\frac{7}{3}}&0&0&1\cr
5 & {\frac{5}{6}}&0&-{\frac{8}
{27}}&-7&0&0
}
$}

\caption{Bisection-convex polygon of figure~\ref{fig:bisectionconvex}(a) and its local bisections matrix $R_P$ of $r_{i,j}$ values.}
\label{fig:rho-vals}
\end{figure}

The entries in the first upper diagonal   are the reciprocals of the corresponding entries of $\Delta_P$ multiplied by half the area $A_P=30$ of the polygon (figure~\ref{fig:triangles}). E.g. $r_{0,1}=(1/9)\times(15)=5/3$. And now~$r_{0,2}$, for example, is calculated as
$r_{0,2}=(r_{0,1}-1)\Delta_{0,1}/\Delta_{0,2}=(2/3)\times 9/(-27/8)=-16/9.$ Proposition~\ref{prop:skip} is applied in the final row of $R_P$: from the first upper diagonal entry which  (mod 6) is $r_{5,0}= 5/6$ we must skip to $r_{5,2}$ which, by proposition~\ref{prop:skip} is given by $r_{5,2}=\left(r_{5,0}-1\right)\Delta_{5,0}/\Delta_{5,2}= -8/27.$

\subsection{Pivots}
For a bisection-convex polygon $P$, in each row of the local bisections matrix $R_P$ there is precisely one entry in the interval $(0,1]$  and each such entry corresponds to a bisecting chord of the polygon. We are careful to specify a half-open interval because, apart from the diagonal and subdiagonal zeros, there are zeros arising from collinear vertices (as excluded from equation~(2) in proposition~\ref{prop:feedforward}). In our example, in figure~\ref{fig:rho-vals}, $r_{5,1}=0$ because $\Delta_{5,1}=0$ due to the fact that vertices $5,1$ and~$2$ are collinear.

And then there are zeros which correspond to  bisecting chords. This latter case signifies a vertex-vertex chord bisecting the area of~$P$ so that, say, $r_{i,j}=1$ identifies the second vertex (anticlockwise) of edge $[j,j+1]$ and $r_{i,j+1}=0$ identifies the first vertex of edge $[j+1,j+2]$. This 0-1 pair will have a corresponding pair of entries in row~$j+1$ with $r_{j+1,i-1}=1$ and $r_{j+1,i}=0$. In  figure~\ref{fig:rho-vals}, the vertex-vertex bisecting chord from vertex~$0$ to vertex~$4$ corresponds to the four entries in $R_P$: $r_{0,3}=1, r_{0,4}=0, r_{4,5}=1, r_{4,0}=0$. Each vertex-vertex bisecting chord features four times in $R_P$, twice as a unit and twice as a zero. We are going to restrict ourselves to one of each so that we can associate each bisecting chord with a unique entry in $R_P$. These entries thus identify the chords that will pivot to give us bisection angles on demand.

We want a definition, then, of those  entries in $R_P$ which identify `pivot' chords, that is, the $r_{i,j}$ that signal  that triangle $\Delta_{i,j}$ is a bisecting chord. The easy case is $r_{i,j}\in (0,1)$: these qualify automatically. Unit entries are more subtle because they come in pairs and we must unambiguously choose one from each pair. In the next subsection we are going to order our bisecting chords to form a half-circle and we will do this by scanning the entries of $R_P$ along its antidiagonals. This turns out to be also a convenient way of distinguishing between unit entries. The antidiagonals are specified by ordering the row and column pairs as follows:
\begin{equation}
\label{eqn:antidiag}
((i,m-i\!\!\mod n),i=0,\ldots, n-1), m=0,\ldots,n-1.
\end{equation}
The $R_P$ entry in row~$i$ and column~$j$ is the $i$-th entry on antidiagonal $m=i+j$. Now two unit entries corresponding to the same bisecting chord are $r_{i,j}=1$ and  $r_{j+1,i-1}=1$ and therefore lie on the same antidiagonal. We will choose  the unit appearing first on this antidiagonal. The bisecting chord originating at the other end of the same vertex-vertex pair we choose to represent  by its zero entry in $R_P$. This will be the second of two zeros on the same antidiagonal which, in addition, immediately (mod $n$) follows  a unit entry.
Referring back to figure~\ref{fig:rho-vals}, the two unit entries lie on the same  antidiagonal $m=3$ at $i=0$ and $i=4$. We choose the $r_{0,3}=1$ and $r_{4,0}=0$ to represent the two orientations of this vertex-vertex chord.

We translate this back into row-column `coordinates' in  the following definition:

\begin{Def}
\label{def:pivot}
 A row, column pair $(i,j)$ for the local bisections matrix $R_P$ will be called a {\bf pivot} if $0<r_{i,j}<1$, or if
$$ r_{i,j}=\left\{\begin{array}{lcl}
 1 & \mbox{and} & i<j, \mbox{ or}\\
 0 & \mbox{and} &  j<i  \mbox{ and } r_{i,j-1}=1
 \end{array}\right.
$$
\end{Def}
(Strictly speaking, the first of the two unit entries $r_{i,j}=1$ and  $r_{j+1,i-1}=1$ satisfies $i<j+1$ but we cannot have a unit entry with $i=j$ so we can choose the neater $i<j$ in our definition).

We have
\begin{Prop}
\label{prop:pivots}
If $P$ is a bisection-convex polygon then $R_P$ has a unique pivot in each row and on each antidiagonal.
\end{Prop}
{\bf Proof.} Each pivot in $R_P$ corresponds to a distinct bisecting chord of $P$. It follows immediately from theorem~\ref{Thm:bisection-convex} that there are $n$ pivots and that each row contains one and only one pivot. Now the ordered pairs in expression~(\ref{eqn:antidiag}) may be taken to be the edges in a complete directed graph on $n$ vertices. Specifically, each vertex has a directed edge to every vertex including itself. The antidiagonals partition these edges into equal subsets of size~$n$. Embed the graph in the plane as a regular $n$-gon. Then for each antidiagonal, any two edges are either disjoint or  are concurrent: they  are the same edge oriented oppositely. By definition only one of a pair of concurrent edges may represent a pivot, and by lemma~\ref{lm:intersecting}, no two bisecting chords may be represented by disjoint edges. Therefore, no antidiagonal in $R_P$ can contain more than one pivot. Since there are $n$ bisecting chords each antidiagonal of~$P$ must contain at least one pivot.\newline\mbox{\ }\QED

It would be nice if the converse of proposition~\ref{prop:pivots} held  to give a neat characterisation of bisection convexity but unfortunately this is not the case. Certainly, for the non-bisection-convex polygon of figure~\ref{fig:bisectionconvex}(b) the first row of $R_P$ is found to have three values in the interval $(0,1)$. But if the vertex~$2$ is moved to point $(2,4)$ the result is still a non-bisection-convex polygon but now every row and antidiagonal of the new local bisections matrix has a single $(0,1)$ entry. We  resort to theorem~\ref{Thm:bisection-convex} as our means of checking bisection-convexity, the details being given in subsection~\ref{subsection:bc-testing}.

\subsection{Bisecting chords: ordering as direction vectors}
We assume an affirmative answer from a test for bisection-convexity. If $r_{i,j}$ is a pivot then the chord from vertex $i$ to point $\vc{v}_j\ang{r_{i,j}}=\vc{v}_j\left(1-r_{i,j}\right)+\vc{v}_{j+1}r_{i,j}$ is a bisecting chord of  polygon $P$.
Our strategy is to take these bisecting chords and order them into a half-circle of direction vectors, implementing the proof of lemma~\ref{lm:circling}, as illustrated in figure~\ref{fig:sectorvectors}. And, which amounts to the same thing, each pair of consecutive chords must join the same pair of edges, as in figure~\ref{fig:sectors}.
Given two direction vectors forming a sector of the half-circle it is then straightforward to determine whether a vector~$\vc{u}$ lies within this sector.
Thus, given a slope vector $\vc{u}$ we determine which bisecting chord to pivot in order to produce a bisecting line in the direction of~\vc{u}.

 Implementing the ordering in the proof of lemma~\ref{lm:circling} needs a little care and this is where the antidiagonal ordering of the last subsection is useful. If we scan matrix $R_P$ in row-major order  then the consecutive pivots  will not normally appear in an order which respects the pivoting conditions illustrated in figure~\ref{fig:sectors}. For example in  figure~\ref{fig:rho-vals}  the first two bisecting chords encountered correspond to $r_{0,3}=1$ and $r_{1,5}=1/6$. However, we do not want to take the corresponding bisecting chords as consecutive direction vectors around our half-circle  because, although they originate in a  common edge $[0,1]$, their destination edges are different: $[4,5]$ and $[5,0]$, respectively, contrary to the scheme of figure~\ref{fig:sectors}. The correct ordering in which to take the originating vertices of our bisecting chords is $0,4,5,1,2,3$. This is the ordering which partitions the half-circle into sectors in figure~\ref{fig:sectorvectors}.

Geometrically, ordering the chords is simple enough: follow a chord to its head (destination) edge; follow this edge anticlockwise until the next chord is encountered; follow this new chord to its opposite edge, and so on. Translating this into vector arithmetic, involving points of intersection and changes of direction, might be accepted as a necessary implementation chore. However, it turns out that our antidiagonal scan does the implementation for us.

\begin{Prop}
\label{prop:torus}
If the   bisecting chords of a bisection-convex polygon $P$ are ordered as they appear as pivots during an antidiagonal  scan of the elements of~$R_P$, according to expression~(\ref{eqn:antidiag}) then

\vspace{-.15in}
 \addtolength{\itemsep}{-1.4\itemsep}
\begin{enumerate}
\item the endpoints of any pair of consecutive bisecting chords will occupy the same pair of polygon edges;
\item  treated as direction vectors, the ordered bisecting chords may be oriented to partition the half-circle.
\end{enumerate}
\end{Prop}

{\bf Proof.}
\begin{enumerate}
\item Suppose that entry $(i,j)$ is a pivot, so that $\Delta_{i,j}$ is a bisecting chord. Let $r_{x,y}$ be the next pivot encountered in our scan. By proposition~\ref{prop:pivots} this will lie on antidiagonal $m=i+j+1$. Also, by lemma~\ref{lm:intersecting}, $\Delta_{x,y}$ must have nonzero intersection with $\Delta_{i,j}$. There are two triangles $\Delta_{x,y}$ that satisfy $x+y=i+j+1$ and have nonzero intersection with $\Delta_{i,j}$: $x=i+1, y=j$, and $x=j+1, y=i$. These are the only possible  triangles that give intersection because as the value of $x$ increases anticlockwise around $P$, so the value of $y$ decreases: for $i+1<x<j+1$ we have $y<j$. And similarly, if $y$ increases anticlockwise then $x$ decreases. Finally we observe that
    both $\Delta_{i+1,j}$ and $\Delta_{j+1,i}$ share with $\Delta_{i,j}$ three vertices out of the two edges $[i,i+1], [j,j+1]$, so in either case we have two consecutive bisecting chords whose endpoints share the same pair of edges
\item Start  with the pivot in the first row of $R_P$, that is, the first pivot corresponds to bisecting chord $\Delta_{0,i}$, for some $i$. Locate the pivots, according to definition~\ref{def:pivot}, as they appear on successive antidiagonals. Because we start  with a chord directed away from vertex~0, we may orient the chords consistently with this increase, as direction vectors, by reversing any chord $\Delta_{x,y}$ for which $y<x$.  Since consecutive bisecting chords intersect, the angular distance from chord $\Delta_{0,i}$ increases anticlockwise. Therefore the antidiagonal ordering of pivots implements the proof of lemma~\ref{lm:circling} and produces a half-circle.
\end{enumerate}
\QED

\section{Applying the tabulated polygon data}
\subsection{Creating a bisecting chord of the required slope}
We are now in a position to construct, for a given bisection-convex polygon, a picture like figure~\ref{fig:sectorvectors}. That is, we have a set of bisecting chords and a representation of these chords as sectors of the half-circle, specified as direction vectors. A given slope vector $\vc{u}$ is now easy to locate as lying within one of the sectors, or as being identical in slope to one of the direction vectors: compare the normal vector to $\vc{u}$ to consecutive direction vectors until the angle formed changes from acute to obtuse, or becomes zero.

Thus we may  proceed to put calculations on to the picture in figure~\ref{fig:sectors}(b), as depicted in figure~\ref{fig:find-t}.
 \begin{figure}[h]
\centerline{\includegraphics[scale=0.4]{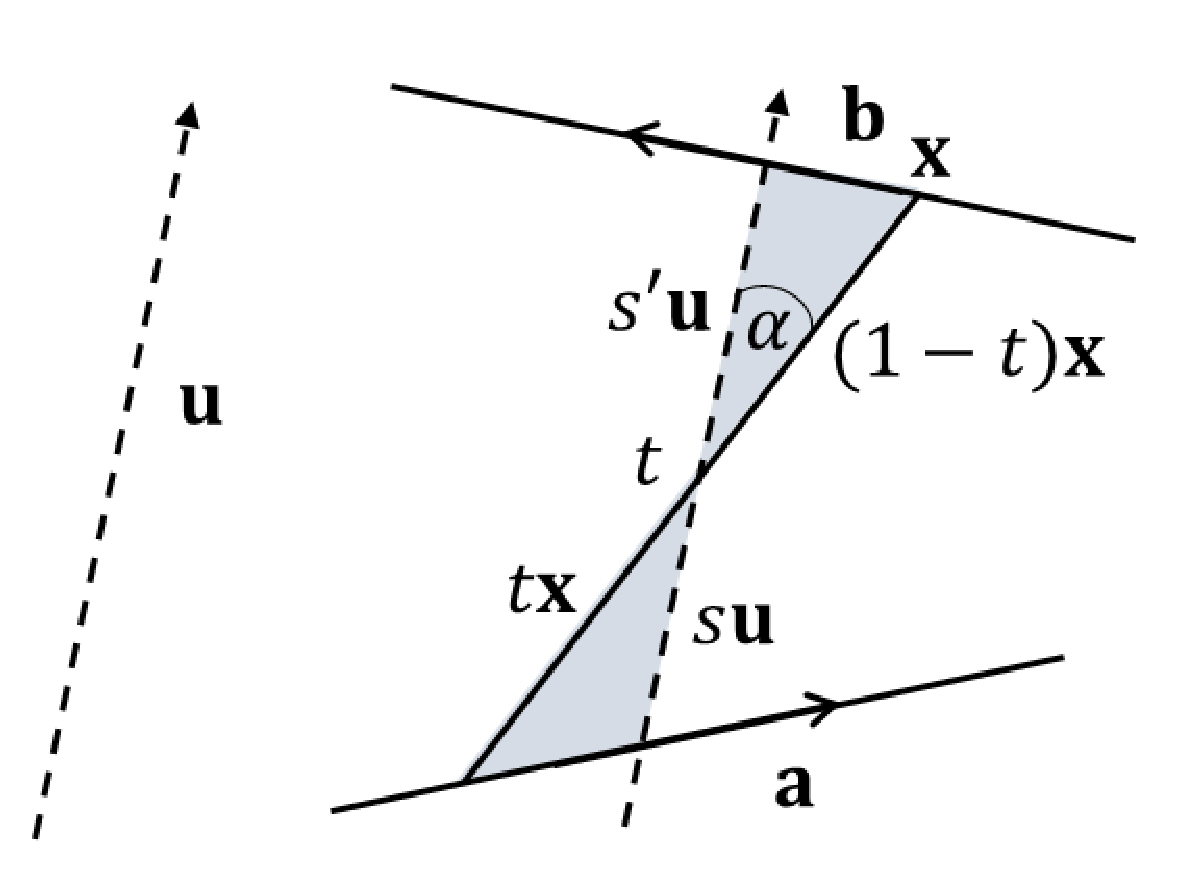}}

\caption{Finding the point of intersection of slope vector \vc{u} with bisecting chord \vc{x}}
\label{fig:find-t}
\end{figure}

In figure~\ref{fig:find-t}, the aim is to find the point of intersection of (translated) vector $\vc{u}$ which will given equal areas to the two triangles having bases on edge vectors $\vc{a}$ and $\vc{b}$, respectively. Since we know the length of bisecting chord $\vc{x}$ we can assert that this point of intersection divides $\vc{x}$ in the proportion $t$ to $(1-t)$, with $t$ unknown. The values $s$ and $s'$ are the scaling factors of $\vc{u}$ which give the distance from the point of intersection to edges $\vc{a}$ and $\vc{b}$, respectively. We do not, a priori, know how $s$ and $s'$ are related. However, when the two triangles in question have equal areas we can write
\begin{eqnarray}
\frac12|s\vc{u}||t\vc{x}|\sin\alpha&=&\frac12|s'\vc{u}||(1-t)\vc{x}|\sin\alpha\nonumber\\
\mbox{whence }s'&=&\frac{t}{1-t}s.
\label{eqn:sdash}
\end{eqnarray}
Now to find the value of $t$, it will suffice to write down two equations, and can do this with the two vector calculations which say, for each of our two triangles, following the two sides on $\vc{x}$ and $\vc{u}$ gives a vector orthogonal to the normal to the triangle base:
\begin{eqnarray*}
(t\vc{x}-s\vc{u})\cdot \vo{a}&=&0\\
(-s'\vc{u}+(1-t)\vc{x})\cdot\vo{b}&=&0.
\end{eqnarray*}
Eliminating $s'$ using equation~\ref{eqn:sdash} and then eliminating $s$ solves for $t$:
$$\frac{1-t}{t}=\sqrt{\frac{(\vc{x}\cdot\vo{a})(\vc{u}\cdot\vo{b})}{(\vc{x}\cdot\vo{b})(\vc{u}\cdot\vo{a})}}=\sqrt{T}, \mbox{ say}.$$
So \begin{equation}
t=\left(1+\sqrt{T}\,\right)^{-1}.
\label{eqn:tval}
\end{equation}

We show the calculation of $t$ in action in figure~\ref{fig:bisected}. The slope vector $\vc{u}$, shown as a dotted line at the origin, is found to be located in the final sector of the half-circle, after  the fifth bisecting chord, as ordered in the previous section. The fifth bisecting chord is accordingly pivoted, about the appropriately calculated point, marked with a circle, to give the direction of $\vc{u}$.
 \begin{figure}[h]
\centerline{\includegraphics[scale=0.6]{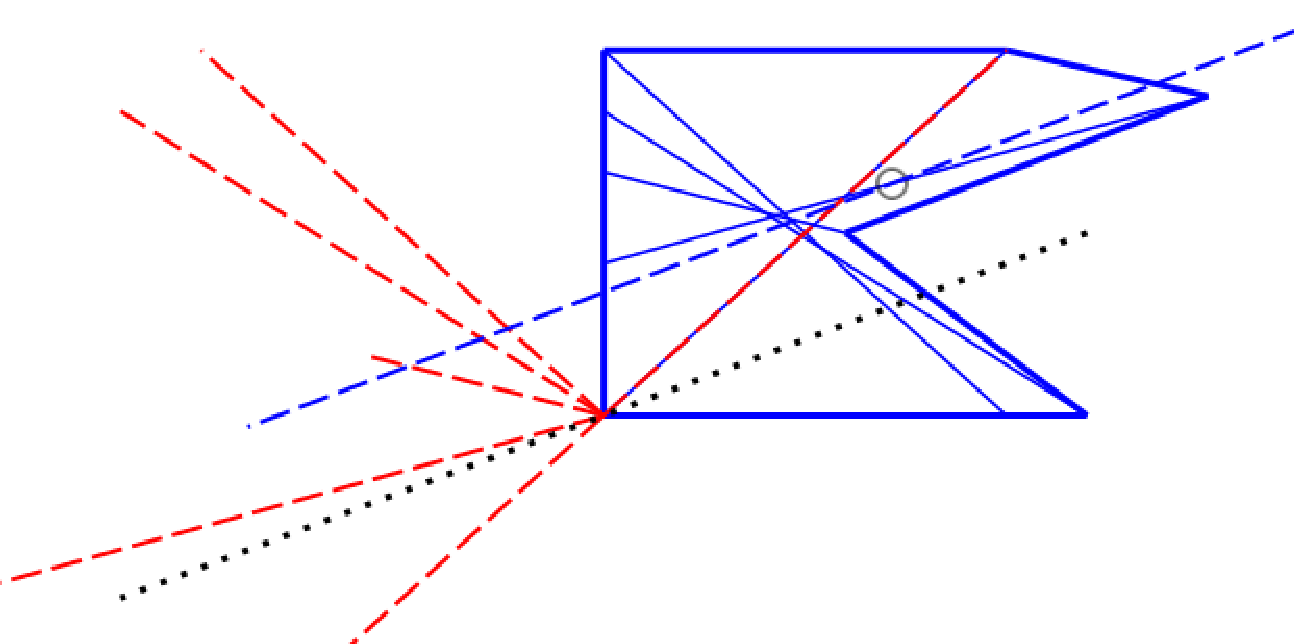}}

\caption{Vector $\vc{u}$ (dotted) translated to bisect the given polygon}
\label{fig:bisected}
\end{figure}

\subsection{Testing for bisection convexity}
\label{subsection:bc-testing}
Theorem~\ref{Thm:bisection-convex} in subsection~\ref{sub:bisecting-chords-id} suggests an algorithm for deciding if a given $n$-vertex polygon $P$ is bisection-convex: try to construct a bisecting chord at each vertex. If we are successful $n$ times then $P$ is bisection-convex. By proposition~\ref{prop:pivots} this test can be implemented in terms of the local bisections matrix $R_P$: a necessary condition for bisection-convexity is that each row must contain exactly one pivot. As we observed, this is not a sufficient condition: we still have to check that the chord specified by the pivot contains no points exterior to the polygon. In other words we must check that the chord intersects no edges of $P$ other than at its end points.  Although this is routine vector arithmetic it is quite laborious; luckily we can supply a speed-up by re-using the  triangle areas matrix $\Delta_P$.

The text-book check for line intersection may be stated in reference to figure~\ref{fig:cross-prod} in which the vertices are labelled with position vectors while $\vc{r}$ is the direction vector corresponding to the bisecting chord at $\vc{v}_i$. The polygon edge joining vertices $\vc{x}$ and $\vc{y}$ will intersect with the straight line passing through $\vc{v}_i$ in direction $\vc{r}$ if and only if $\vc{x}$ and $\vc{y}$ lie in different halfplanes in relation to this straight line. This occurs if and only if the two cross products $(-\vc{v}_i+\vc{x})\times\vc{r}$ and $(-\vc{v}_i+\vc{y})\times\vc{r}$ have different signs. So the strategy is to calculate the series of cross products $(-\vc{v}_i+\vc{v}_{i+k})\times\vc{r}$, $k=1,\ldots,n-1$ and look for sign changes. (This is the same idea as testing for convexity except that in that case we take cross products of consecutive edges.) As in subsection~\ref{subsection:local-bisections}, these are cross products in the plane so they are effectively scalars, $(a,b)\times (c,d)=ad-bc$.
 \begin{figure}[h]
\centerline{\includegraphics[scale=0.3]{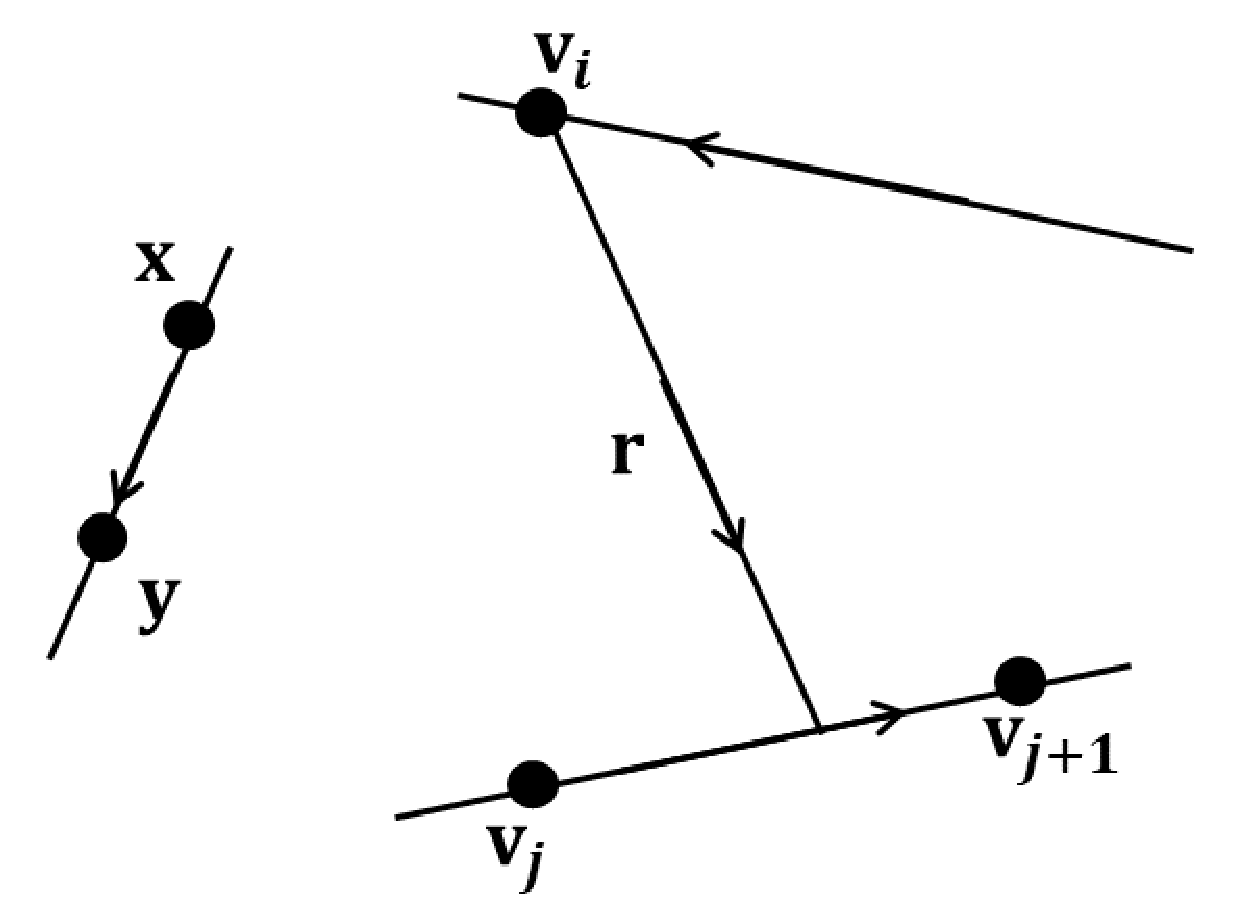}}

\caption{Testing polygon edge $xy$ for intersection with $\vc{r}$.}
\label{fig:cross-prod}
\end{figure}

The translation to triangle areas is due to the following lemma.
\begin{Lm}
\label{lm:cross-prod}
For bisecting chord $\vc{r}=-\vc{v}_i+(1-r_{i,j})\vc{v}_j+r_{i,j}\vc{v}_{j+1}$ we have
$$(-\vc{v}_i+\vc{x})\times\vc{r}=2r_{i,j}\left(\Delta_{i,j}-\Delta_{x,j}\right)+2\Delta_{i,x,j},$$
where $\Delta_{i,x,j}$ is the area of the triangle on vertices $\vc{v}_i, \vc{x},\vc{v}_j$.
\end{Lm}
{\bf Proof.}
For simplicity we will refer to vectors $\vc{v}_i$ by their subscripts. We again have recourse to the basic cross product properties: $i\times i=0$ and $i\times j=-j\times i$, which we will now further simplify by writing as $i^2=0$ and $ij=-ji$. Then we may recall the relevant equivalencies between triangle areas and cross products thus:
\begin{eqnarray*}
2\Delta_{i,j}&=&(-i+j)\times(-i+(j+1))=ij-i(j+1)+j(j+1)\\
2\Delta_{x,j}&=&(-x+j)\times(-x+(j+1)=xj-x(j+1)+j(j+1)\\
2\Delta_{i,x,j}&=&(-i+x)\times(-i+j)=ix-ij+xj.
\end{eqnarray*}
Now we compute
\begin{eqnarray*}
(-i+x)\times\vc{r}&=&-i\times(-i+j-r_{i,j}j+r_{i,j}(j+1))\\
&&+x\times(-i+j-r_{i,j}j+r_{i,j}(j+1))\\
&=&-ij+r_{i,j}ij-r_{i,j}i(j+1)+r_{i,j}j(j+1)-r_{i,j}(j(j+1)\\
&&+ix+xj-r_{i,j}xj+r_{i,j}x(j+1)-r_{i,j}j(j+1)+r_{i,j}j(j+1)\\
&=&2r_{i,j}\left(\Delta_{i,j}-\Delta_{x,j}\right)+2\Delta_{i,x,j}.
\end{eqnarray*}
\QED

 \begin{figure}[h]
\centerline{\includegraphics[scale=0.36]{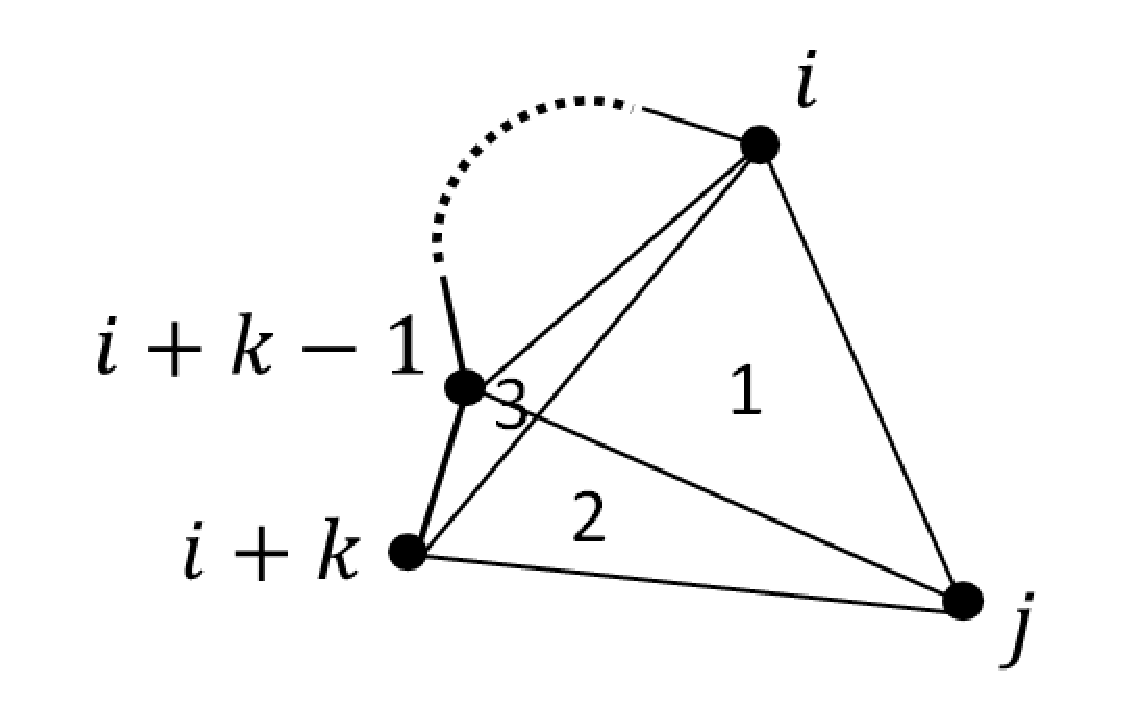}}

\caption{Recurrence in Lemma~\ref{lm:recurrence}}
\label{fig:recurrence}
\end{figure}
It remains to express the triangle $\Delta_{i,x,j}$ in terms of the entries in the $\Delta_{i,j}$ matrix. Given that our interest is in checking consecutive edges around the polygon we may rely on the following recurrence:
\begin{Lm}
\label{lm:recurrence}
For $k=1,\ldots,n-1$,
$$\Delta_{i,i+k,j}=\Delta_{i,i+k-1,j}+\Delta_{j,i+k-1}-\Delta_{i,i+k-1}.$$
\end{Lm}
{\bf Proof.}
The three triangles in figure~\ref{fig:recurrence} correspond to the three terms on the right-hand-side of the Lemma. Areas~1 and~2, less area~3 gives area $\Delta_{i,i+k,j}$ as required.
\QED

We remark that such arguments as given in the above proof gloss over the impact of non-convex sections of the polygon. We leave it to the reader to reassure themselves that the clockwise oriented triangles that arise, computing as negative area, resolve any potential discrepancies.

 From the preceding lemmas and with the preceding notation we have:
  \begin{Lm}
\label{lm:triangles} Define the sequence $F_k$, $k\geq 0$, by
$$F_0=0 \mbox{ and, for }k\geq 1, F_k=F_{k-1}+\Delta_{j,i+k-1}-\Delta_{i,i+k-1}.$$
Then for $k=1,\ldots,n-1$,
$$\left(-\vc{v}_i+\vc{v}_{i+k}\right)\times\vc{r}=2r_{i,j}\left(\Delta_{i,j}-\Delta_{i+k,j}\right)+2F_k.$$
\end{Lm}
\QED

Now our test for bisection-convexity is the following:
 \begin{Thm}
\label{thm:triangles} Let $P$ be an $n$-vertex polygon with triangle areas matrix $\Delta_P$ and local bisections matrix $R_P$. For distinct vertices $i$ and $j$ of $P$,
 the triangle  $\Delta_{i,j}$ is a bisecting chord if and only if
\begin{enumerate}
\item in $R_P$, entry $r_{i,j}$ is a pivot and is the only pivot in row $i$;
\item with the sequence $F_k$  defined as in lemma~\ref{lm:triangles}, the sequence of values
$$2r_{i,j}\left(\Delta_{i,j}-\Delta_{i+k,j}\right)+2F_k, k=1,\ldots,n-1,$$
exhibits exactly one change of sign.
\end{enumerate}
\end{Thm}
\QED

We can apply the two tests of theorem~\ref{thm:triangles} to figure~\ref{fig:bisectionconvex}(b). Vertex~0 fails the first test: $r_{0,1},r_{0,2}$ and $r_{0,4}$ are all pivots (the first two are the non-local bisections shown in the figure, the non-local bisection corresponding to $r_{0,4}$ was not shown). All other rows of the local bisections matrix have a unique pivot and therefore pass test~1. But just two vertices pass the sign condition of test~2: vertex~1 records four pluses followed by a minus: four edges following vertex~1 and then edge $[5,0]$ which crosses the bisecting chord. And vertex~5 records a plus followed by four minuses: the edge $[5,0]$ is immediately followed by the edge crossing the bisecting chord.

\section{Summary and concluding remarks}
We have completed our explanation of how to bisect the area of a given bisection-convex polygon with a straight line cut in a given direction. It seems worthwhile to summarise the steps involved, and this is done in table~\ref{tab:bisection-steps}. The main work is calculating the matrices in steps~1 and~2. This data allows us to check whether the given polygon is bisection-convex and logically this is step~3 since if the answer is No then we should abandon our approach in favour of Shermer's algorithm. If the answer is Yes then the data for step~4 is read off directly from matrix~$R_P$ and thereafter we are in the happy position of being able to do rapid repeated bisections.

\begin{table}[h]
\begin{tabular}{c|l|c}\hline
1 & Calculate triangle areas matrix $\Delta_P$ & section 3.1\\
2 & Calculate local bisections matrix $R_P$ & section 3.3 \\
3 & Confirm polygon is bisection-convex & section 4.2 \\
4 & Form half-circle of bisecting chord angles & section 3.5 \\
5 & Locate given angle in half-circle &  section 4.1 \\
6 & Pivot chosen bisecting chord to given angle & section 4.1 \\
7 & Repeat steps 5 and 6 as required & \\ \hline
\end{tabular}
\caption{Schema for bisection-convex all-angles polygon bisection}
\label{tab:bisection-steps}
\end{table}

By `rapid repeated bisections' we mean sublinear time in the worse case, per bisection: if the required bisection angles arrive distributed uniformly at random in the half-circle then step~5 in our table is a $\log n$ binary search; step~6 is, of course, constant time. If the required bisection angle is increased or decreased monotonically (a typical application being `ham sandwich' bisection of two polygons, e.g. \cite{Diaz}) then  successive bisections will be made in constant time.

As for our two matrices, $\Delta_P$ and $R_P$, they require, a priori, a quadratic amount of time to construct. Theorem~\ref{Thm:DeltaRank} offers the  prospect of driving the recursive procedures of proposition~\ref{prop:feedforward} and theorem~\ref{thm:triangles} entirely from just three initial rows of $\Delta_P$. This might reduce the computational complexity of steps~1,~2 and~3 in table~\ref{tab:bisection-steps} but seems to indicate a great deal of complicated book-keeping!

As it stands, our method is competitive only for  Fechtor-Pradines' `not-too-non-convex' polygons and for a need for repeated bisection angles, ideally monotonically increasing or decreasing. Nevertheless, our local bisections matrix is defined for any simple polygon and the bisection-convexity test of theorem~\ref{thm:triangles} works on a vertex-by-vertex basis. This means that, even in a very non-convex polygon, we may be able to identify sectors of the half-circle where our method still works. For example, in figure~\ref{fig:bisectionconvex}(b) the two bisecting chords $\Delta_{1,5}$ and $\Delta_{5,0}$ describe a sector of more than $\tau/15$ radians within which we can produce bisections on demand.

Moreover, we produced our method based on the idea of `local' bisections on the understanding that bisection-convexity would qualify them as `global'. Global, halfplane, bisection is essential in some contexts, such as the aforementioned ham-sandwich application. But our approach has wider applicability. Suppose that  figure~\ref{fig:bisectionconvex}(b) represents a field to be divided between two farmers. We can calculate that the corresponding local bisections matrix has values $r_{0,1}=31/40$ and $r_{0,2}=9/43$. Then the locally bisecting chord $\Delta_{0,1}$ (that is, the chord which joins vertex~0 to edge $[1,2]$) seems to be a perfectly equitable proposition: one farmer (given region $P^+(0,1,31/40)$) has twice as many angles to negotiate but she cannot argue with the area these angles enclose. The same might be said for locally bisecting chord $\Delta_{2,0}$. Such local bisections even have a `creative accounting' side to them: the locally bisecting chord $\Delta_{0,2}$ gives a larger triangular area to the farmer receiving $P^-(0,2,9/43)$ but the second farmer might be given in compensation the triangle, exterior to the polygon, cut off at vertex~2 by the chord: the local bisection property guarantees that this has the same area as the extra wedge given to farmer no.~1. We have lost sight of our motivation of all-angles bisection (unless our farmers are very anxious about prevailing winds) but it would be nice to find our two matrices of polygon data  had more milage (or perhaps we mean acreage) in them than just ham sandwiches.

\end{document}